\documentclass[a4paper]{amsart}

\usepackage{a4,pstricks,bbm}
\usepackage{amsmath,amssymb,latexsym}

\usepackage{amssymb,latexsym}

\usepackage{enumerate}

\makeatletter

\@namedef{subjclassname@2010}{%

  \textup{2010} Mathematics Subject Classification}

\makeatother

\makeatletter

\@namedef{subjclassname@2010}{%

  \textup{2010} Mathematics Subject Classification}

\makeatother

\newtheorem{theorem}{Theorem}[section]

\newtheorem{corollary}[theorem]{Corollary}

\newtheorem{lemma}[theorem]{Lemma}
\newtheorem{proposition}[theorem]{Proposition}

\theoremstyle{definition}

\newtheorem{conjecture}[theorem]{Conjecture}

\numberwithin{equation}{section}

\def\Z{\mathbb Z}

\begin{document}


\baselineskip=17pt



\title{On dilates sum}
\author {Y. O. Hamidoune}
\address{UPMC, Univ Paris 06,\\ 4 Place Jussieu,
75005 Paris, France.}
\email{hamidoune@math.jussieu.fr}
\author {J. Ru\'e}
\address{LIX, \'Ecole Polytechnique ,  91128 Palaiseau-Cedex, France}
\email{juan.jose.rue@upc.edu}

\thanks{The second author is supported by the European Research Council under
the European Community's 7th Framework Programme, ERC grant agreement no 208471 - ExploreMaps project.}

\maketitle

\begin{abstract} Let $A$ be a finite nonempty set of integers.
An asymptotic estimate of several dilates sum size was obtained by Bukh. The unique known exact
bound concerns the sum $|A+k\cdot A|,$  where $k$ is a prime and $|A|$ is large. In its full generality, this bound
is due to Cilleruelo, Serra and the first author.

Let $k$ be an odd prime and assume that
$|A|>8k^{k}.$ A corollary to our main result states that
$|2\cdot A+k\cdot A|\ge (k+2)|A|-k^2-k+2.$
Notice that  $|2\cdot P+k\cdot P|=(k+2)|P|-2k,$ if $P$ is an arithmetic progression.
 \end{abstract}

\section{Introduction}

Let $A,\ B$ be finite nonempty sets of real numbers. The {\em Minkowski sum} of $A$ and $B$ is defined as
$$A+B=\{a+b \ : \ a\in A\  \mbox{and}\ b\in B\}.$$
The inequality $|A+B|\ge |A|+|B|-1$ is an easy exercise, that we shall use without any reference.
For a real number $r,$ the $r$-\emph{dilate} of $A$ 
is the set $r\cdot A=\{ra : a\in A\}$.  Lower bounds for the size of dilates sums appeared in different contexts. $\L$aba and Konyagin~\cite{laba} investigated the  sum $A+\lambda \cdot A$ (where $\lambda$ is a transcendental  number) in connection with well-distributed planar sets distances.   Dilates sums also appeared in the proofs of sum-product results in finite fields by Garaev~\cite{garaev}, and by Katz and
Shen~\cite{KS}. The sum of two dilates appeared in the work of Nathanson, O'Bryant, Orosz, Ruzsa and Silva on binary linear forms~\cite{NOOR}.  Also, they were used  by Bukh \cite{bukh0} in connection with a problem of Ruzsa.

From now on,   we assume that $A$ is a nonempty set of integers. Dilates sum of the form $A+3\cdot A$ were investigated independently by Bukh \cite{bukh} and  by
Cilleruelo, Silva and  Vinuesa~\cite{javier}. More recently, the authors of~\cite{joy} proved that for an odd prime $k$,  $|A+k\cdot A|\ge (1+k)|A|- (k+1)^2/4$ for $|A|$ sufficiently large.
Let  $m_1, \cdots, m_j$ be integers
with $\gcd(m_1, \dots, m_j)=1$, Bukh proved in~\cite{bukh} that
$$|m_1\cdot A+\dots +m_j\cdot A|\ge (|m_1|+\dots +|m_j|)|A|-{\small o}\left(|A|\right).$$
%


Bukh's result suggests the following:

\begin{conjecture}
For a set of integers $Z$ with $\gcd(Z)=1$ and for every nonempty set of integers $A,$ there is an absolute constant $c$ such that
  $$\left|\sum _{m\in Z} m\cdot A\right|\ge \left(\sum _{m\in Z}|m|\right)|A|-c.$$
  \end{conjecture}

 In the present work, we prove the above conjecture for $Z=\{2,k\},$ where $k$ is an odd prime. For simplicity, we will not consider
 negative dilates, but the reader will certainly observe that our approach works in this case.


In section 2, we present some easy and known lemmas that we need. In section 3, we prove some intermediary results needed in our induction arguments. One of the results of this section states that
$|n\cdot A+m\cdot A|\ge 4|A|-4$, where $m$ and $n$ are coprime integers. This result is a counterpart of a lemma by Nathanson~\cite{Nathanson} stating that $|A+2\cdot A|\ge 3|A|-2$.
Let $k$ be an odd prime. By a $k$-component of a set $X\subset \Z,$ we shall mean the trace on $X$ of some congruence classes modulo
$k$. In section 4, we investigate the marginal set $(2\cdot C+k\cdot A)\setminus (2\cdot C+k\cdot C),$
where $C$ is a $k$-component of $A$. In section 5, we prove that $|2\cdot A+k\cdot A|\ge (k+2)|A|-4k^{k-1}.$

Assuming that $0\in A,$  $\gcd(A)=1$,  $|A|>8k^{k}$ and that $A$ has a $k$-component involving at most $k-1$   $k^2$-components, we show that
$$|2\cdot A+k\cdot A|> (k+2)|A|.$$
 Readers interested in  the description of  sets reaching equality could quite likely  use this result, since it shows that the objective function $|2\cdot A+k\cdot A|$ achieves its minimum on   structured sets, for $|A|$ large.  Let $X$ be a finite set of integers with
$|X|>8k^{k}.$ We conclude Section 5 by an  easy consequence of our
  main result,   stating  that
$|2\cdot X+k\cdot X|\ge (k+2)|X|-k^2-k+2.$ As an exercise, the reader could prove that
 $|2\cdot P+k\cdot P|=(k+2)|P|-2k,$ if $P$ is an arithmetic progression. Observe that for $k=3,$ our bound differs at most by $4$ from the best possible one.

\section{Preliminaries and terminology}

In this paper, we consider sums of dilates of a finite set of integers.
 The next known lemma shows that the size of a dilates sum remains invariant if we replace $A$ by an affine transform of it.

\begin{lemma}~\cite{bukh} Let $A$ be a finite set of integers and let $r,s,u,v$ be non-zero integers. Then
  \begin{itemize}
    \item $|r\cdot ( A+v)+s\cdot(A+v)|= |r\cdot  A+s\cdot A|,$
    \item $|r\cdot (u\cdot A)+s\cdot(u\cdot A)|= |r\cdot  A+s\cdot A|.$
  \end{itemize}
\label{prehistorical}
\end{lemma}

\begin{proof}
 We have clearly
\begin{align*}
|r\cdot ( A+v)+s\cdot( A+v)|&= |r\cdot A+s\cdot A+(rv+sv)|
\\ &= |r\cdot A+s\cdot A|.
\end{align*}
as $|A+w|=|A|$. We also have
\begin{align*}
|r\cdot (u\cdot A)+s\cdot(u\cdot A)|&=
| (ru)\cdot A+(su)\cdot A|\\ &= |u\cdot (r\cdot A)+u\cdot (s\cdot A)|
\\ &= |r\cdot A+s\cdot A|.\end{align*}
\end{proof}

Let $A$ be a finite set of integers. The intersection of  $A$  with a congruence class modulo $n$ will be called a $n$-{\em component}. By a {\em decomposition} modulo $n,$ we mean a partition of $A$ into its $n$-components.
The number of $n$-components of $A$ will be denoted by $c_n(A).$
We shall say that $A$ is {\em $n$-full} if $c_n(A)=n$. The set $A$ is $n$-{\em semi-full} if every $n$-component $C$ of $A$ satisfies  $c_{n^2} (C)=n.$

\begin{lemma} \label{prefull}If $A$ is $n$-full, then $\gcd(A)$ is coprime to $n$. Moreover, $\frac{1}{\gcd(A)}\cdot A$ is $n$-full.
  \end{lemma}
  \begin{proof} There is an  $u\in A$ such that $u \equiv 1 \pmod n.$ As $\gcd(A)$ divides $u$, then $\gcd(\gcd(A),n)$ divides both $u$ and $n$, hence it divides $1$. Additionally, $\gcd(A)$ is invertible modulo $n$, and thus $c_{n}\left(\frac{1}{\gcd(A)}\cdot A\right)=c_{n}(A).$
    \end{proof}

\section{Tools}

 \begin{lemma} Let $A$ and $B$ be finite sets of integers and let $m, n$ be coprime integers. Let ${\mathcal C}$ denote the
  set of $m$-components of $A.$ Then $n\cdot A+m\cdot B=\bigcup _{C\in {\mathcal C}}
   n\cdot C+m\cdot B$ is a decomposition modulo $m.$
\label{bigcomp}
 \end{lemma}
 \begin{proof} Clearly $n\cdot C+m\cdot B \equiv n\cdot C \pmod m,$  for any  $C\in {\mathcal C}.$
 The result follows now since $n\cdot C$ and $n\cdot T$ are necessarily incongruent modulo $m$, for distinct components
 $C,T\in {\mathcal C}.$ \end{proof}

The next lemma is basic in our approach.
  \begin{proposition} Let $A$ and $B$ be finite sets of integers and let $m, n$ be coprime integers. Then $|n\cdot A+m\cdot B|\ge c_n(B)|A|+c_m(A)|B|-c_m(A)c_n(B).$
\label{basic}
 \end{proposition}
 \begin{proof}

Let ${\mathcal A}$ be the set $m$-components  of $A$ and let ${\mathcal B}$ be the set $n$-components of $B.$
We claim that if $M_1,M_2\in  {\mathcal A}$ and $N_1,N_2\in  {\mathcal B}$ such that $(M_1,N_1)\neq (M_2,N_2)$, then $(n\cdot M_1+m\cdot N_1)\cap (n\cdot M_2+m\cdot N_2)= \emptyset.$ Suppose the contrary and take $a_i\in M_i$ and $b_i\in N_i,$ $1\le i \le 2$ with
$na_1+mb_1=na_2+mb_2.$ Thus, $n(a_1-a_2)=m(b_1-b_2).$
Since $m$ is coprime to $n,$ we have $b_1-b_2\equiv 0 \pmod n$ and $a_1-a_2\equiv 0 \pmod m.$
In particular, $M_1=M_2$ and $N_1=N_2$, a contradiction.

Therefore, using Lemma~\ref{bigcomp} we have
\begin{align*}
|n\cdot A+ m\cdot B|&=\left|\bigcup_{M\in {\mathcal A} ;\ N\in {\mathcal B}} n\cdot M+m\cdot N\right|
\ge \sum_{M\in {\mathcal A} ;\ N\in {\mathcal B}}  |M|+|N|-1\\
&=  c_n(B)|A|+c_m(A)|B|-c_m(A)c_n(B).
\end{align*}
\end{proof}

\begin{corollary} Let $2\leq n< m$ be coprime integers. Let $A$ be a finite set of integers. Then
$|n\cdot A+m\cdot A|\ge 4|A|-4.$
\label{nathanson}
 \end{corollary}
 \begin{proof}
 The result holds clearly if $|A|=1.$ Assume that $|A|\ge 2.$
Put $B=\frac{1}{\gcd(A)}\cdot A$. 
Put $r=c_m(A)$ and $s=c_n(A).$ Without loss of generality we may assume that $r\ge s.$ Observe that that $2\le r \le |A|$.
By Lemma~\ref{prehistorical} and Proposition \ref{basic}, we have \begin{align*}
|n\cdot A+m\cdot A|&=|n\cdot B+m\cdot B|\\
&\ge 4|B|-4+(r+s-4)|B|-rs+4\\
&\ge 4|B|-4+(r+s-4)r-rs+4\\
&\ge 4|B|-4+(r-2)^2\ge 4|A|-4.\end{align*}
\end{proof}

  \begin{corollary} Let $m$ be an odd integer. Let $A$ be a $m$-full finite set of integers. Then
$|2\cdot A+m\cdot A|\ge (m+2)|A|-2m.$
In particular,
$|2\cdot A+m\cdot A|\ge (m+2)|A|-2mc_m(A),$  if $A$ is $m$-semi-full.
\label{subfull1}
 \end{corollary}
 \begin{proof}
 The first part is a direct consequence of Proposition~\ref{basic}. For the second part, take a $m$-decomposition of $A$, namely  $A=\bigcup _{i\in I} A_i$.
 Take an arbitrary element  $i\in I.$ Since $A$ is $m$-semi-full,  $A_i$ can be affinely
 transformed into an $m$-full subset, by Lemma \ref{prefull}. By the first part of this corollary, $|2\cdot A_i+m\cdot A_i|\ge (m+2)|A_i|-2m.$
 By Lemma~\ref{bigcomp}, $2\cdot A_i+m\cdot A_i$ and $2\cdot A_j+m\cdot A_j$ belong to different congruence classes  modulo $m$ for $i\neq j$.
   By Lemma~\ref{bigcomp}, $|2\cdot A+m\cdot A|\ge \sum _{ i \in I}|2\cdot A_i+m\cdot A_i|\ge \sum _{ i \in I} (m+2)|A_i|-2m=(m+2)|A|-2mc_{m}(A).$
\end{proof}

\section{Dilates sum size}

Let $A$ be a finite set of integers and let $C$ be a component of $A.$
  The $C$-{\em marginal} set is defined as $$M_{C}=(2\cdot C+k\cdot A)\setminus (2\cdot C+k\cdot C).$$

 We start by proving a bound for marginal sets size in the semi-full case.

 Let $A$ be a finite set of integers and let ${\mathcal C}$ denotes the set of $k$-components of $A.$
 We shall denote by $\Gamma$ the graph of the order relation $x<y$ defined on the set $L=\{\min (C) : C\in {\mathcal C}\}.$ Recall that for a given element $x\in L$, $\Gamma(x)=\{y\in L : y> x\}$ and $\Gamma^{-}(x)=\{y\in L : y<x\}$.
 For any component $C\in {\mathcal C},$ we shall write ${M_C}^-=\{x\in M_C : x<\min (2\cdot C+k\cdot C)\}$ and
  $M_C^+=\{x\in M_C : x>\max (2\cdot C+k\cdot C)\}$. With the above notations, we formulate the next lemma.

\begin{lemma}
\label{order}
$\sum _{C\in {\mathcal C}} |M_C|\ge  (|{\mathcal C}|-1)| {\mathcal C}|.$
 \end{lemma}

 \begin{proof}
Clearly $$2\cdot \Gamma ^-(\min (C))+k\cdot \min (C) \subset M_C^{-}.$$
In particular, $ |\Gamma ^-(\min (C))|\leq |M_C^-|.$ Similarly, $ |\Gamma (\max (C))|\leq |M_C^+|.$
Therefore,
\begin{align*}
\sum _{C\in {\mathcal C}} |M_C| &= \sum _{C\in {\mathcal C}}  |M_C^-|+|M_C^+|\\
&\ge \sum _{C\in {\mathcal C}}  |\Gamma ^-(\min (C))|+\sum _{C\in {\mathcal C}}  |\Gamma ^-(\min (C))|\\
&=2(|{\mathcal C}|-1)| {\mathcal C}|/ 2)=(|{\mathcal C}|-1)| {\mathcal C}|,
\end{align*}
since $\sum _{C\in {\mathcal C}}  |\Gamma ^-(\min (C))|=\sum _{C\in {\mathcal C}}  |\Gamma ^-(\min (C))|$ is the total number of arcs in the order relation, which is obviously $(|{\mathcal C}|-1)| {\mathcal C}|/ 2$.
 \end{proof}

We shall say that $C$ is {\em faithful} if $|M_C|\ge |C|.$  Faithful sets are  satisfactory for induction proofs, as we shall see in the next section. For a subset $X$ of an abelian group $G,$ we write $\pi (X)=\{x\in G : x+X=X\}.$ We recall that $|\pi(X)|$ is a divisor of $|X|.$

The next lemma gives conditions implying faithfulness.
\begin{lemma} Let $k$ be a prime and let $0\in A$ be a finite set of integers with $\gcd(A)=1.$ Let $C$ be a non $k$-semi-full
component of $A$ and let $C'\neq C$ be another $k$-component of $A.$ Then $|M_C|\ge |C'|.$
Moreover,
$C$ is faithful if one of the following conditions holds:
\begin{itemize}
  \item There is another component with size not less than $ |C|.$
  \item $C$ is  non $2$-full.
 \end{itemize}

\label{tl}
 \end{lemma}

  \begin{proof}
Let $\phi:\mathbb{Z}\rightarrow \mathbb{Z}/k^2\mathbb{Z}$ be the projection.
Since  $|\pi (C)|$ divides $k^2$ and
$|\phi (C)|$ and since $|\phi (C)|<k,$ we have necessarily $\pi (\phi (C))=\{0\}.$ Assuming  that $\phi(2\cdot C+k\cdot C)=\phi(2\cdot C+k\cdot C')$,
we have  $\phi(k\cdot C)=\phi(k\cdot C'),$ and hence $C\equiv C'$ modulo $k.$
It follows that $C=C',$ a contradiction.

Thus, $M _{C}$ contains  $2\cdot C_0+k\cdot C',$ where $C_0$ is some $k^2$-component of $C$, and
hence $|M _{C}|\ge |C'|.$ Assume now that $C$ is non-$2$-full. Take an odd element $v\in A.$ Clearly, $2\cdot C+k v$
is a $C$-marginal set.\end{proof}

 \section{Large sets of integers}
 The next result is our first step.

\begin{theorem} If $A$ is a finite set of integers, then
$$|2\cdot A+k\cdot A|\ge (k+2)|A|-4k^{k-1}.$$
\label{main0}
 \end{theorem}
\begin{proof} Without loss of generality, we may take $0\in A$ and $\gcd(A)=1$. 
We shall prove by induction that for all $2\le s\le k,$ we have
$$|2\cdot A+k\cdot A|\ge (s+2)|A|-4k^{s-1}.$$
For $s=2$, this bound is weaker than the one obtained by Corollary~\ref{nathanson}.

Assume now that $2 < s \leq k$ and that the result holds for $s-1$. Take a $k$-decomposition of $A$, namely  $A=\bigcup _{i\in I} A_i$. We shall write $M_i=M_{A_{i}}.$ Put $F=\{i\in I : A_i \text{ is} \  k-\text{full}  \}$
and $E=I\setminus F.$ Notice that $|E|+|F|=|I|\le k.$ 

Assume first that  $$\sum _{i\in E} |M_i|\ge \sum _{i\in E} |A_i|.$$

By Lemma~\ref{bigcomp},
 $2\cdot A_i+m\cdot A$ and $2\cdot A_j+m\cdot A$ belong to different congruence classes modulo $k,$ for $i\neq j.$  By Corollary~\ref{subfull1}, for every $i\in F$, $|2\cdot A_i+k\cdot A_i|\ge (k+2)|A_i|-2k$.  Using the last relations,
the induction hypothesis applied for every $i\in E,$ we have 
\begin{align*}
|2\cdot A+k\cdot A|\ge&\,
\sum _{i\in E} (|2\cdot A_i+k\cdot A_i|+|M_{i}|)+
\sum _{i\in F} |2\cdot A_i+k\cdot A_i|\\
\ge & \sum _{i\in E} ((1+s)|A_i|-4k^{s-2})+ \sum _{i\in E} |A_i|+\sum _{i\in F} ((k+2)|A_i|-2k)\\
\ge&\, \sum _{i\in I} (s+2)|A_i|-4|I|k^{s-2}+2k|F|\left(2k^{s-3}-1\right)\\
\ge & (s+2)|A|-4k^{s-1},
\end{align*}
and the result holds.

Assume now that $$\sum _{i\in E} |M_i|< \sum _{i\in E} |A_i|.$$ In particular, we have $|E|\ge 1.$ We must have $|E|=1,$
otherwise we take a derangement (permutation without fixed element) $\sigma$ of $E.$ By Lemma~\ref{tl}, $|M_i|\ge |A_{\sigma (i)}|,$  for every $i.$ We get a contradiction by summing. Put $E=\{e\}$ and take $f\in F.$ By Lemma~\ref{tl}, $A_e$ is $2$-full,  $|A_e|> |A_f|$  and $|M_e|\ge |A_f|.$
 By Lemma \ref{basic}, $|2\cdot A_f+k\cdot A_e|\ge 2|A_e|+k|A_f|-2k$.

The idea here is to estimate the component of $2\cdot A+k\cdot A$ containing $2\cdot A_f+k\cdot A_f$, using the sum  $2\cdot A_f+k\cdot A_e.$ Using Lemma~\ref{bigcomp} we have,
\begin{align*}
|2\cdot A+k\cdot A|\ge&
|2\cdot A_e+k\cdot A_e|+|M_{e}|+|2\cdot A_f+k\cdot A_e|+
\sum _{i\in F\setminus f} |2\cdot A_i+k\cdot A_i|\\
\ge& (1+s)|A_e|-4k^{s-2}+|A_f|+k|A_f|+2|A_e|-2k+ \sum _{i\in F\setminus f}  ((k+2)|A_i|-2k)\\
>& (s+2)|A_e|-4k^{s-2}+ \sum _{i\in F}  ((s+2)|A_i|-2k)\\
\ge&(s+2)|A|-4k^{s-1}+2|F|k(2k^{s-3}-1)\\
\ge&(s+2)|A|-4k^{s-1}.
\end{align*}
\end{proof}
Our main result is the following one:
\begin{theorem} Let $A$ be a finite set of integers with
$0\in A,$  $\gcd(A)=1$ and   $|A|>8k^{k}.$
 If $A$ has a $k$-component involving at most $k-1$  distinct $k^2$-components, then
$$|2\cdot A+k\cdot A|> (k+2)|A|.$$
\label{main1}
 \end{theorem}

 \begin{proof}
 Let ${\mathcal C}$ denote the set of $k$-components of $A$ and let $C$ a non $k$-semi-full component of $A.$
 Take a component $N$ with a maximal cardinality. Clearly $|N|>8k^{k-1}.$ Assume first that $|M_C|\ge |N|.$
 By Theorem~\ref{main0} and Lemma~\ref{bigcomp} we have
 \begin{align*}
|2\cdot A+k\cdot A|\ge&\,
|2\cdot C+k\cdot C|+|M_{C}|+|2\cdot (A\setminus C)+k\cdot (A\setminus C)|\\
>& (k+2)|C|-4k^{k-1}+8k^{k-1}+(k+2)|A\setminus C|-4k^{k-1}\\
=&\,(k+2)|A|,
\end{align*}

Assume now that $|M_C|\ge |N|$ and put $c=c_k(A).$ By Lemma \ref{tl}, $C=N$ and $C$ is $2$-full. Also we may assume that $C$ is the unique
non $k$-semi-full component of $A.$ Take now a  $k$-semi-full component of $A,$ say $T$ and put $A'=A\setminus (C\cup T).$
By Lemma \ref{basic}, $|2\cdot T+k\cdot C|\ge 2|C|+k|T|-2k$.

 Thus using Corollary~\ref{subfull1} and  Lemma~\ref{bigcomp} we have
\begin{align*}
|2\cdot A+k\cdot A|\ge&\,
|2\cdot C+k\cdot C|+|T|+|2\cdot T+k\cdot C|+(k+2)|A'|-2(c-2)k
\\
\ge & (k+2)|C|-4k^{k-1}+|C|+(k+2)|T|-2k+ (k+2)|A'|-2(c-2)k\\
>&(k+2)|A|+4k^{k-1}-2k(k-1)>(k+2)|A|,
\end{align*}
and the result holds.
 \end{proof}

We can now prove the  following lower bound on $|2\cdot A+k\cdot A|$:

\begin{corollary} If $A$ is a finite set of integers with
$|A|>8k^{k},$ then
$$|2\cdot A+k\cdot A|\ge (k+2)|A|-k^2-k+2.$$
\label{main2}
 \end{corollary}

 \begin{proof}
  By Lemma \ref{prehistorical}, we may take $0\in A$ and $\gcd(A)=1$. By Corollary~\ref{subfull1}, the result holds if $A$ is $k$-full. Assume that $A$ is non $k$-full and let ${\mathcal C}$ denote the set of $k$-components of $A.$
Put $j= | {\mathcal C}|.$
 By Theorem \ref{main1}, we may assume that $A$ is $k$-semi-full. By Corollary \ref{subfull1} and  Lemma \ref{bigcomp}, we have
 \begin{align*}
|2\cdot A+k\cdot A|=&
\sum _{C\in {\mathcal C}}|2\cdot C+k\cdot C|+|M_{C}|\\
\ge& \sum _{C\in {\mathcal C}}((k+2)|C|-2k)+j(j-1)\\
=&(k+2)|A|-j(2k-j+1)
\end{align*}
Therefore  $|2\cdot A+k\cdot A|\ge (k+2)|A|-k^2-k+2$, and the result holds.
 \end{proof}


\bibliography{bukh}
\bibliographystyle{abbrv}










%

\end{document}